\documentclass{amsart}

\usepackage{amsmath}
\usepackage{amsfonts}
\usepackage{amsthm} 
\usepackage{amssymb}

\usepackage{graphicx} 

\usepackage{color} 

\newtheorem{thm}{Theorem}[section]
\newtheorem{prop}[thm]{Proposition}

\newtheorem{lem}[thm]{Lemma}

\newtheorem{defn}[thm]{Definition}

\numberwithin{equation}{section} 

\makeatletter 
\@namedef{subjclassname@2010}{2010 Mathematics Subject Classification} 
\makeatother 

\begin{document}

\title{On Singular Points and Oscillatory Integrals} 

\author{Toshio NAGANO and Naoya MIYAZAKI} 

\address{Department of Liberal Arts, Faculty of Science and Technology, Tokyo University of Science, 2641, Yamazaki, Noda, Chiba 278-8510, JAPAN} 

\address{Department of Mathematics, Faculty of Economics, Keio University, Yokohama, 223-8521, JAPAN} 

\thanks{The first author was supported by Tokyo University of Science Graduate School doctoral program scholarship 
and an exemption of the cost of equipment from 2016 to 2018 
and would like to thank to Professor Minoru Ito for 
giving me an opportunity of studies and preparing the environment.} 

\dedicatory{Dedicated to Emeritus Professor Hideki Omori on his 80th birthday} 



\date {June 1, 2019} 

\begin{abstract} 
In this note, we generalize the Fresnel integrals 
using oscillatory integral, 
and then we obtain an extention of the stationary phase method. 
\end{abstract} 

\maketitle 

\section{Introduction} 

In the present note, 
we study a method for an asymptotic 
expansion of the following integral 
\begin{align} 
I(\lambda) 
:= \int_{\mathbb{R}^{n}} e^{i \lambda \phi(x)} a (x) dx \notag 
\end{align} 
as a real positive parameter $\lambda \to \infty$. 
For the purpose we first give a summary of the original method 
relating to theory of asymptotic expansion in \S 2. 
In particular, the stationary phase method is known as 
a method for an asymptotic expansion of an oscillatory integral 
of a phase function with a {\bf non-degenerate critical point}. 

In \S 3, we extend the Fresnel integrals 
by changing of a path 
for integration 
in the well-known proof using Cauchy's integral theorem. 
Next, according to oscillatory integral, 
we also obtain further generalization of the Fresnel integrals. 

Furthermore, in \S 4, 
we establish finer generalization of the stationary phase method. 

As mentioned as above, 
the several proofs of the Fresnel integrals are known (\cite{Sugiura} I p.326, II p.85, 245, etc.). 
In particular, 
we are interested in 
the proof of applying Cauchy's integral theorem 
to a function $e^{-iz^{2}}$ on the domain surrounded 
by a fan with the center at the origin of Gaussian plane (\cite{Ito-Komatsu}). 
By changing the fan used in the proof, 
we can generalize the Fresnel integrals in the following way: 
\begin{align} 
I_{p,q}^{\pm} 
:= \int_{0}^{\infty} e^{\pm ix^{p}} x^{q-1} dx 
= p^{-1} e^{\pm i\frac{\pi}{2} \frac{q}{p}} 
\varGamma \left( \frac{q}{p} \right), \notag 
\end{align} 
where $\varGamma$ is the Gamma function. 
This equalities hold 
for $p>q>0$ (Lemma \ref{Generalized the Fresnel integrals}). 
Moreover, 
by making a sense 
of these integrals via oscillatory integral, 
we obtain 
\begin{align} 
\tilde{I}_{p,q}^{\pm} 
:= \lim_{\varepsilon \to +0} \int_{0}^{\infty} 
e^{\pm ix^{p}} x^{q-1} \chi (\varepsilon x) dx 
= p^{-1} e^{\pm i\frac{\pi}{2} \frac{q}{p}} 
\varGamma \left( \frac{q}{p} \right) , \notag 
\end{align} 
where $\chi \in \mathcal{S}(\mathbb{R})$ with $\chi(0) = 1$. 
This equalities hold for $p > 0$ and $q > 0$ (Theorem \ref{th01} (i)). 
These results can be considered as an extension of 
the case of $\lambda = q-1$ and $\xi = 1$ 
in the Fourier transform of Gel'fand-Shilov 
generalized function $\mathcal{F}[{x_{+}}^{\lambda}](\xi)$ 
with $\lambda \in \mathbb{C} \setminus \{ -1 \}$ (\cite{Gel'fand-Shilov} p.170.): 
\begin{align} 
\mathcal{F}[{x_{+}}^{q-1}](1) 
:= \lim_{\tau \to +0} \int_{0}^{\infty} e^{ix} x^{q-1} e^{-\tau x} dx 
= e^{i\frac{\pi}{2}q} \varGamma (q) (1 + i0)^{-q}, \notag 
\end{align} 
where $\mathrm{Re} \tau > 0$. 
By our results, 
we can become to treat the oscillatory integrals for a phase function 
with a {\bf degenerate critical point} expressed by positive real power. 
And then, 
using the result obtained as above, 
we give an extention of the stationary phase method in one variable 
(Theorem \ref{Extension of the Stationary Phase Method} (iv)). 

To the end of \S1, we remark 
notation which will be used in this note: 

$\alpha = (\alpha_{1},\dots,\alpha_{n}) \in (\mathbb{Z}_{\geq 0})^{n}$ 
is a multi-index with a length 
$| \alpha | = \alpha_{1} + \cdots + \alpha_{n}$, 
and then, we use 
$x^{\alpha} = x_{1}^{\alpha_{1}} \cdots x_{n}^{\alpha_{n}}$, 
$\alpha! = \alpha_{1}! \cdots \alpha_{n}!$, 
$\partial_{x}^{\alpha} 
= \partial_{x_{1}}^{\alpha_{1}} \cdots \partial_{x_{n}}^{\alpha_{n}}$ 
and 
$D_{x}^{\alpha} = D_{x_{1}}^{\alpha_{1}} \cdots D_{x_{n}}^{\alpha_{n}}$, 
where 
$\partial_{x_{j}} = \frac{\partial}{\partial x_{j}}$ 
and $D_{x_{j}} = \frac{1}{i} \partial_{x_{j}}$ 
for $x = (x_{1}, \dots, x_{n})$. 

$C^{\infty}(\mathbb{R}^{n})$ is 
the set of complex-valued $C^{\infty}$ functions on $\mathbb{R}^{n}$. 
$C^{\infty}_{0}(\mathbb{R}^{n})$ is 
the set of all $f \in C^{\infty}(\mathbb{R}^{n})$ with compact support. 
$\mathcal{S}(\mathbb{R}^{n})$ is 
the Schwartz space of rapidly decreasing $C^{\infty}$ 
functions on $\mathbb{R}^{n}$, 
that is, the Fr\'{e}chet space of all $f \in C^{\infty}(\mathbb{R}^{n})$ 
such that $\sup_{x \in \mathbb{R}^{n}}
 | x^{\beta} \partial_{x}^{\alpha} f (x) | < \infty$ 
for any multi-indecies $\alpha, \beta \in (\mathbb{Z}_{\geq 0})^{n}$. 

$O$ means a Landau's symbol, that is, 
$f(x) = O(g(x))~(x \to a)$ if $|f(x)/g(x)|$ 
is bounded as $x \to a$ for functions $f$ and $g$. 

\section{Preliminary}

In this section, 
we recall the oscillatory integrals and 
the original stationary phase method. 
\begin{defn} 
Let $\lambda \geq 1$ 
and let $\phi$ be a real-valued $C^{\infty}$ function on $\mathbb{R}^{n}$ 
and $a \in C^{\infty}(\mathbb{R}^{n})$. 
If there exists 
\begin{align} 
\tilde{I}_{\phi}[a](\lambda) 
:= Os\text{-}\int_{\mathbb{R}^{n}} e^{i \lambda \phi(x)} a (x) dx 
:= \lim_{\varepsilon \to +0} \int_{\mathbb{R}^{n}} 
e^{i \lambda \phi(x)} a (x) \chi (\varepsilon x) dx \notag 
\end{align} 
independent of $\chi \in \mathcal{S}(\mathbb{R}^{n})$ with $\chi(0) = 1$ and 
$0 < \varepsilon < 1$, 
then we call $\tilde{I}_{\phi}[a](\lambda)$ an oscillatory integral 
where we call $\phi$ 
$($resp. $a$$)$ a phase function $($resp. an amplitude function$)$. 
\end{defn} 

If we suppose a certain suitable conditions for $\phi$ and $a$, 
then we can show $\tilde{I}_{\phi}[a](\lambda)$ exists independent of 
$\chi$ and $\varepsilon$. 
The fundamental properties are the following 
(cf. \cite{Kumano-go} p.47.): 
\begin{prop} 
\label{chi epsilon} 
Assume that $\chi \in \mathcal{S}(\mathbb{R}^{n})$ with $\chi(0) = 1$. 
Then 
\begin{enumerate} 
\item[(i)] 
$\chi(\varepsilon x) \to 1$ uniformly on any compact set in $\mathbb{R}^{n}$ as $\varepsilon \to +0$. 
\item[(ii)] 
For each multi-index $\alpha \in (\mathbb{Z}_{\geq 0})^{n}$, 
there exists a positive constant $C_{\alpha}$ independent of $0 < \varepsilon < 1$ 
such that 
for $x \in \mathbb{R}^{n}$ 
\begin{align} 
| \partial_{x}^{\alpha} (\chi(\varepsilon x)) | \leq C_{\alpha} (1+|x|^{2})^{-|\alpha|/2}. \notag 
\end{align} 
\item[(iii)] 
For any multi-index $\alpha \in (\mathbb{Z}_{\geq 0})^{n}$ with $\alpha \ne 0$, 
$\partial_{x}^{\alpha} \chi(\varepsilon x) \to 0$ uniformly in $\mathbb{R}^{n}$ as $\varepsilon \to +0$. 
\end{enumerate} 
\end{prop} 

Next we recall the Fresnel integrals (cf. \cite{Ito-Komatsu} p.23.). 
\begin{prop}
The following integrals, which called Fresnel integrals, hold. 
\begin{align} 
\int_{-\infty}^{\infty} e^{\pm ix^{2}} dx = \sqrt{\pi} e^{\pm i \frac{\pi}{4}}, \notag 
\end{align} 
where double signs $\pm$ in same order. 
\end{prop}

Next 
we summarize the Fourier transforms of rapidly decreasing 
$C^{\infty}$ functions. 
\begin{defn} 
\label{the Fourier transform} 
Let $f \in \mathcal{S}(\mathbb{R}^{n})$. 
Then 
we define by $\widehat{f}=\mathcal{F}[f]$ the Fourier transform of $f$ as 
\begin{align} 
\mathcal{F}[f](\xi) 
:= \frac{1}{(2\pi)^{\frac{n}{2}}} \int_{\mathbb{R}^{n}} 
e^{-i\langle x,\xi \rangle} f(x) dx, \notag 
\end{align} 
where $\langle x,\xi \rangle :=\sum_{k=1}^{n}x_{k} \xi _{k}$ 
for $x=(x_{1},\dots,x_{n})$ and $\xi =(\xi_{1},\dots,\xi _{n}) 
\in \mathbb{R}^{n}$. 
\end{defn} 

If $A$ is a real symmetric non-singular $n \times n$ matrix, 
then the Fourier transform of $e^{i (1/2) \langle Ax,x \rangle}$ is given 
in the following way (cf. \cite{Hormander01}, \cite{Hormander02}, 
\cite{Duistermaat01}, \cite{Grigis-Sjostrand}, \cite{Fujiwara1}).
\begin{prop} 
\label{the Fourier transform of e_ix2} 
Let $A$ be a real symmetric non-singular $n \times n$ matrix with ``$p$" 
positive and ``$n-p$" negative eigenvalues. 
\begin{enumerate} 
\item[(i)] 
If $A = \pm 1$ for $n=1$, 
then 
\begin{align} 
\mathcal{F}[e^{\pm i\frac{1}{2}x^{2}}](\xi) 
= e^{\pm i\frac{\pi}{4}} e^{\mp i\frac{1}{2} \xi^{2}}, \notag 
\end{align} 
where double signs $\pm, \mp$ in same order. 
\item[(ii)] 
If $n \geq 1$, 
then 
\begin{align} 
\mathcal{F}[e^{i \frac{1}{2} \langle Ax,x \rangle}](\xi) 
= \frac{e^{i \frac{\pi}{4} \mathrm{sgn}A}}{|\det A|^{\frac{1}{2}}} 
e^{-i \frac{1}{2} \langle A^{-1}\xi,\xi \rangle}, \notag 
\end{align} 
where $\mathrm{sgn}A := p-(n-p)$. 
\end{enumerate}
\end{prop}

By Propsition \ref{the Fourier transform of e_ix2}, 
we can obtain an asymptotic expansion of the oscillatory integral 
with a non-degenerate quadratic phase $\phi(x) = (1/2) \langle Ax,x \rangle$ 
in the following way. 
\begin{prop} 
\label{quadratic phase} 
Suppose that $\lambda \geq 1$, 
$a \in \mathcal{S}(\mathbb{R}^{n})$ 
and $A$ is a real symmetric non-singular $n \times n$ matrix. 
Then, 
there exists a positive constant $C$ such that 
for any $N \in \mathbb{N}$
\begin{align} 
&\int_{\mathbb{R}^{n}} e^{i\frac{1}{2} \lambda \langle Ax,x \rangle} a(x) dx \notag \\ 
&= (2\pi)^{\frac{n}{2}} \frac{e^{i\frac{\pi}{4} 
\mathrm{sgn}A}}{|\det A|^{\frac{1}{2}}} 
\sum_{k=0}^{N-1} \frac{1}{k!} 
\Big( -i\frac{1}{2} \langle A^{-1} D_{x},D_{x} \rangle \Big)^{k} \Big|_{x=0} 
a(x) \lambda^{-k-\frac{n}{2}} + R_{N}(\lambda) \notag 
\end{align} 
and 
\begin{align} 
|R_{N}(\lambda)|
\leq (2\pi)^{\frac{n}{2}} \frac{C}{|\det A|^{N+\frac{1}{2}}} 
\frac{1}{N!} \Big( \sum_{|\alpha| \leq 2(N+n)} 
\int_{\mathbb{R}^{n}} |\partial_{x}^{\alpha} a(x)| dx \Big) 
\lambda^{-N-\frac{n}{2}}. \notag 
\end{align} 
\end{prop}

In order to treat more general cases of the phase function, 
we prepare the following two lemmas. 
The first one is the Morse lemma 
(\cite{Milnor}, \cite{Duistermaat01}, \cite{Grigis-Sjostrand}, \cite{Fujiwara1}). 

\begin{lem} 
Let $\phi$ be a real-valued $C^{\infty}$ function on 
a neighborhood of $\bar{x}$ in $\mathbb{R}^{n}$  
such that $\bar{x}$ is an only non-degenerate critical point of $\phi$, 
that is, 
if and only if $\nabla \phi (\bar{x})=0$, 
and $\det \mathrm{Hess} \phi (\bar{x}) \ne 0$. 
Then 
there exist neighborhoods $U$ of $\bar{x}$ and $V$ of $0$ in $\mathbb{R}^{n}$ 
and $C^{\infty}$ diffeomorphism $\varPhi : V \longrightarrow U$ 
such that $x=\varPhi (y)$ for $x=(x_{1},\dots,x_{n}) \in U$ 
and $y=(y_{1},\dots,y_{n}) \in V$ 
and 
\begin{align} 
\phi (x) - \phi (\bar{x}) 
= \frac{1}{2} (y_{1}^{2} + \cdots + 
y_{p}^{2} - y_{p+1}^{2} - \cdots - y_{n}^{2}), \notag 
\end{align} 
where $\mathrm{Hess} \phi (\bar{x}) 
:= (\partial ^{2} 
\phi (\bar{x})/\partial x_{i} \partial x_{j})_{i,j=1,\dots,n}$ 
is Hessian matrix of $\phi$ 
at $\bar{x}$ with ``$p$" positive and ``$n-p$" negative eigenvalues. 
\end{lem}

The second one is for an estimation of the remainder of an asymptotic expansion for an oscillatory integral (\cite{Hormander02}, \cite{Fujiwara1}). 
\begin{lem} 
\label{Lax's technique} 
Let $\lambda \geq 1$ 
and $a \in \mathcal{S}(\mathbb{R}^{n})$ 
and let $\phi$ be a real-valued $C^{\infty}$ function on $\mathbb{R}^{n}$ 
with $| \nabla \phi (x) | \geq d > 0$ for $x \in \mathrm{supp} a$. 
Then 
for each $N \in \mathrm{N}$, 
there exists a positive constant $C_{N}$ such that 
\begin{align} 
\left| \int_{\mathbb{R}^{n}} e^{i \lambda \phi(x)} a (x) dx \right| 
\leq C_{N}(\lambda d^{2})^{-N}. \notag 
\end{align} 
\end{lem} 

We are now in a position to state the stationary phase method 
(\cite{Hormander01}, \cite{Hormander02}, 
\cite{Duistermaat01}, \cite{Fujiwara1}). 
\begin{thm} 
\label{th00} 
Suppose that $\lambda \geq 1$, $a \in \mathcal{S}(\mathbb{R}^{n})$ 
and $\phi$ is a real-valued $C^{\infty}$ function on a neighborhood of 
$\bar{x}$ in $\mathbb{R}^{n}$  
such that $\bar{x}$ is an only non-degenerate critical point of $\phi$. 
Then 
there exist neighborhoods $U$ of $\bar{x}$ 
and $V$ of $0$ in $\mathbb{R}^{n}$ 
and $C^{\infty}$ diffeomorphism 
$\varPhi : V \longrightarrow U$ such that $x=\varPhi (y)$ 
for $x =(x_{1},\dots,x_{n}) \in U$ and $y =(y_{1},\dots,y_{n}) \in V$, 
and for each $N \in \mathbb{N}$, 
there exists a positive constant $C_{N}$ such that 
\begin{align} 
&\int_{\mathbb{R}^{n}} e^{i\lambda \phi (x)} a(x) dx 
= (2\pi)^{\frac{n}{2}} \frac{e^{i \frac{\pi}{4} \mathrm{sgn} \mathrm{Hess} \phi (\bar{x})}}{|\det \mathrm{Hess} \phi (\bar{x})|^{\frac{1}{2}}} e^{i\lambda \phi(\bar{x})} \notag \\ 
&\times \sum_{k=0}^{N-1} \frac{1}{k!} \Big( -i\frac{1}{2} \langle \mathrm{Hess} \phi (\bar{x})^{-1} D_{y},D_{y} \rangle \Big)^{k} 
\Big|_{y=0} \{ (a \circ \varPhi) J_{\varPhi} \} (y) \lambda^{-k-\frac{n}{2}} + R_{N}(\lambda) \notag 
\end{align} 
and 
\begin{align} 
|R_{N}(\lambda)| \leq C_{N} \lambda^{-N-\frac{n}{2}}, \notag 
\end{align} 
where 
$J_{\varPhi}(y) := 
\det (\partial x_{j}/\partial y_{k})_{j,k=1,\dots,n}$ is a Jacobian of $\varPhi$. 
\end{thm} 

\section{Generalized Fresnel Integrals} 

In this section, 
we consider a generalization of the Fresnel integrals. 

\begin{lem} 
\label{Generalized the Fresnel integrals} 
Assume that $p > q > 0$. 
Then we have 
\begin{align} 
I_{p,q}^{\pm} 
:= \int_{0}^{\infty} e^{\pm ix^{p}} x^{q-1} dx 
= p^{-1} e^{\pm i\frac{\pi}{2} \frac{q}{p}} \varGamma \left( \frac{q}{p} \right), 
\label{I_pq} 
\end{align} 
where $\varGamma$ is the Gamma function 
and double signs $\pm$ in same order. 
\end{lem} 

\begin{proof} 
Apply Cauchy's integral theorem to a homeomorphic function 
$e^{iz^{p}} z^{q-1}$ on the domain with 
the boundary $\sum_{j=1}^{4} C_{j}$ defined by 
$C_{1} := \{ z=r \in \mathbb{C} | 0 < \varepsilon  \leq r \leq R \}$, 
$C_{2} := \{ z=R e^{i\theta} \in \mathbb{C} | 0 \leq \theta \leq \pi /2p \}$, 
$C_{3} := \{ z=-se^{i(\pi /2p)} 
\in \mathbb{C} | -R \leq s \leq -\varepsilon \}$ 
and $C_{4} := \{ z=\varepsilon  e^{-i\tau} 
\in \mathbb{C} | -\pi /2p \leq \tau \leq 0 \}$. 
\end{proof} 

If $q \geq p > 0$, 
we can make a sense of \eqref{I_pq} as oscillatory integrals. 
In order to show this, 
we prepare the following lemma. 
\begin{lem} 
\label{Lax02} 
Assume that $p > 0$ and $q > 0$. 
Let $\chi \in \mathcal{S}(\mathbb{R})$ with 
$\chi(0) = 1$ and $0 < \varepsilon < 1$. 
Then we have 
\begin{enumerate} 
\item[(i)] 
For each $k \in \mathbb{Z}_{\geq 0}$, 
there exist the following independent of $\chi$ and $\varepsilon$ 
\begin{align} 
\lim_{\varepsilon \to +0} \int_{0}^{\infty} 
e^{\pm ix^{p}} x^{q-1} \frac{d^{k}}{dx^{k}} \chi (\varepsilon x) dx. \notag 
\end{align} 
\item[(ii)] 
If $k \ne 0$, 
then 
\begin{align} 
\lim_{\varepsilon \to +0} \int_{0}^{\infty} e^{\pm ix^{p}} x^{q-1} \frac{d^{k}}{dx^{k}} \chi (\varepsilon x) dx = 0. \notag 
\end{align} 
\item[(iii)] 
If $p > q$, 
then 
\begin{align} 
\tilde{I}_{p,q}^{\pm} 
:= Os\mbox{-}\int_{0}^{\infty} e^{\pm ix^{p}} x^{q-1} dx 
= \int_{0}^{\infty} e^{\pm ix^{p}} x^{q-1} dx 
= I_{p,q}^{\pm}. \notag 
\end{align} 
\end{enumerate} 
Double signs $\pm$ in same order. 
\end{lem} 

\begin{proof} 
Dividing this integral by cutoff function 
$\varphi$ 
which has its compact support around the origin in $\mathbb{R}$. 
Repeating application of a differential operator 
$L = \frac{1}{px^{p-1}} \frac{1}{i} \frac{d}{dx}$ to $e^{\pm ix^p}$ 
on the support of $1-\varphi$ and integration by parts 
make the order of integrand descend to be integrable in the sense of Lebesgue.
Application of Lebsgue's dominant convergence theorem 
with Proposition \ref{chi epsilon} completes the proof. 
\end{proof} 

By Lemmas \ref{Generalized the Fresnel integrals} and \ref{Lax02}, 
we obtain the following theorem. 
\begin{thm}
\label{th01}
Assume that $p,q \in \mathbb{C}$. 
\begin{enumerate} 
\item[(i)] 
If $p > 0$ and $q > 0$, 
then 
\begin{align} 
\tilde{I}_{p,q}^{\pm} 
:= Os\mbox{-}\int_{0}^{\infty} e^{\pm ix^{p}} x^{q-1} dx 
= p^{-1} e^{\pm i\frac{\pi}{2} \frac{q}{p}} \varGamma \left( \frac{q}{p} \right). \notag 
\end{align} 
\item[(ii)] 
The $\tilde{I}_{p,q}^{\pm}$ 
can be extended non-zero meromorphic on $\mathbb{C}$ 
with poles of order 1 at $q = -pj$ for $j \in \mathbb{N}$ 
as to $q$ for each $p > 0$, 
and meromorphic on $\mathbb{C} \setminus \{ 0 \}$ with poles of order 1 at 
$p = -q/j$ for $j \in \mathbb{N}$ 
as to $p$ for each $q > 0$ by analytic continuation. 
\end{enumerate} 
Double signs $\pm$ in same order. 
We call $\tilde{I}^{\pm}_{p,q}$ ``generalized Fresnel integrals''. 
\end{thm} 

\begin{proof} 
When $p > q$, applying Lemma \ref{Lax02} (iii) 
and Lemma \ref{Generalized the Fresnel integrals} give (i). 
When $q \geq p$, repeating use of integration by part 
combining with Lemma \ref{Generalized the Fresnel integrals} show (i). 
Using analytic continuation of the Gamma function with (i) proves (ii). 
\end{proof} 

Using the theorem above, 
we can extend the Euler Beta function as follow. 
\begin{prop} 
Assume that $p_{j} > 0$ and $q_{j} 
\in \mathbb{C} \setminus \{ -p_{j}\mathbb{N} \}$ for $j=1,2,3$. 
Let 
\begin{align} 
\tilde{B}^{\pm}(p_{1},p_{2},p_{3};q_{1},q_{2},q_{3}) 
:= e^{\mp i \frac{\pi}{2} \left( \frac{q_{1}}{p_{1}} + 
\frac{q_{2}}{p_{2}} - \frac{q_{3}}{p_{3}} \right)} 
\frac{p_{1} p_{2}}{p_{3}} \frac{\tilde{I}_{p_{1},q_{1}}^{\pm} 
\tilde{I}_{p_{2},q_2}^{\pm}}{\tilde{I}_{p_{3},q_{3}}^{\pm}}. \notag 
\end{align} 
Then 
\begin{align} 
\tilde{B}^{\pm}(1,1,1;q_{1},q_{2},q_{1}+q_{2}) = B(q_{1},q_{2}), \notag 
\end{align} 
where $B(x,y)$ is the Euler Beta function 
and double signs $\pm$ in same order. 
\end{prop}

\section{An Extension of the Stationary Phase Method}

In this section, 
we consider an extension of the stationary phase method in one variable 
according to the generalized Fresnel integrals developed as above. 
\begin{thm} 
\label{Extension of the Stationary Phase Method} 
Let $p > 0$, $\lambda \geq 1$ and $a \in \mathcal{S}(\mathbb{R})$. 
\begin{enumerate} 
\item[(i)] 
If $q > 0$, 
then 
\begin{align} 
I_{p,q}^{\pm}[a](\lambda) 
:= \int_{0}^{\infty} e^{\pm i\lambda x^{p}} x^{q-1} a(x) dx \notag 
\end{align} 
is absolutely convergent. 
\item[(ii)] 
If $q > p$, 
then 
there exists a positive constant $C_{q}$ such that 
\begin{align} 
| I_{p,q}^{\pm}[a](\lambda) | \leq C_{q} \lambda ^{-\frac{q}{p}+1}. \notag 
\end{align} 
\item[(iii)] 
If $p > 0$, 
then for any $N \in \mathbb{N}$ such that $N+1 > p$, 
\begin{align} 
\int_{0}^{\infty} e^{\pm i\lambda x^{p}} a(x) dx 
&= p^{-1} \sum_{k=0}^{N-1} 
e^{\pm i\frac{\pi}{2} \frac{k+1}{p}} \varGamma \left( \frac{k+1}{p} \right) \frac{a^{(k)}(0)}{k!} \lambda ^{-\frac{k+1}{p}} \notag \\ 
&\hspace*{0.3cm}+ O\left( \lambda ^{-\frac{N+1}{p}+1} \right). \notag 
\end{align} 
\item[(iv)] 
If $m \in \mathbb{N}$, 
then for any $N \in \mathbb{N}$ such that $N+1 > m$, 
\begin{align} 
\int_{-\infty}^{\infty} e^{\pm i\lambda x^{m}} a(x) dx 
&= m^{-1} \sum_{k=0}^{N-1} \left( e^{\pm i\frac{\pi}{2} \frac{k+1}{m}} 
+ (-1)^{k} e^{\pm (-1)^{m} i\frac{\pi}{2} \frac{k+1}{m}} \right) \notag \\ 
&\hspace{0.4cm}\times \varGamma \left( \frac{k+1}{m} \right) \frac{a^{(k)}(0)}{k!} \lambda ^{-\frac{k+1}{m}} + O\left( \lambda ^{-\frac{N+1}{m}+1} \right). \notag 
\end{align} 
\end{enumerate} 
Double signs $\pm$ in same order. 
\end{thm} 

\begin{proof} 
(i) Put $f(x) = e^{i\lambda x^{p}} x^{q-1} a(x)$, 
then $f(x) = O(x^{\alpha})~(x \to +0)$ when $\alpha = q-1>-1$ 
and $f(x) = O(x^{\beta})~(x \to \infty)$ when $\beta = q-1-(q+1) < -1.$
This shows (i). 
According to integration by part, we show the second assertion (ii). 
As to (iii), 
first we divide this integral into the principal part 
and the remainder term 
by Taylor expansion of $a$ at $0$. 
And then applying Theorem \ref{th01} (i) 
after using change of variable $y=\lambda^{1/p}x$ in the former 
and applying (ii) in the latter completes the proof. 
(iv)  If the variable is negative, we first change a variable $x = -y$.  
Then (iii) proves (iv). 
\end{proof} 

To the end of the present note, 
we note that we obtain asymptotic expansion of oscillatory integral in several variables with not only the Morse type phase function but also singular types 
for examples $A, E$-phase functions. 
For details, see \cite{Nagano-Miyazaki02}.

\end{document}